\documentclass[a4paper]{amsart}
\usepackage[top=2.5cm, left=1.8cm, right=1.8cm, foot=1.3cm, bottom=3.2cm]{geometry}
\usepackage{amssymb,euscript}
\usepackage{hyperref}

\title{Precompactness for homogeneous Einstein metrics on torus bundles}
\author{Anusha M. Krishnan}
\address[\sc A.\ M. \ Krishnan]{Department of Mathematics, Indian Institute of Technology Bombay \\ Powai, Mumbai - 400076, India}
\email{anushamk@math.iitb.ac.in}

\author{Vincent P. N. Wolff}
\address[\sc V.P.N. Wolff]{Mathematisches Institut, Universit\"at M\"unster, M\"unster, Germany}
\email{vincent.wolff@uni-muenster.de}

\author{Masoumeh Zarei}
\address[\sc M. Zarei]{Fachbereich Mathematik, Universit\"at Hamburg, Hamburg, Germany}
\email{masoumeh.zarei@uni-hamburg.de}


\newcommand{\R}{\mathbb{R}}

\newcommand{\Z}{\mathbb{Z}}
\newcommand{\Q}{\mathbb{Q}}
\newcommand{\J}{\mathcal{J}}
\newcommand{\E}{\mathcal{E}}
\newcommand{\K}{\mathcal{K}}
\newcommand{\F}{\EuScript{F}}
\newcommand{\g}{\mathfrak{g}}
\renewcommand{\a}{\mathfrak{a}}
\renewcommand{\t}{\mathfrak{t}}
\newcommand{\h}{\mathfrak{h}}
\newcommand{\m}{\mathfrak{m}}
\newcommand{\n}{\mathfrak{n}}
\newcommand{\z}{\mathfrak{z}}
\newcommand{\G}{\mathsf{G}}
\newcommand{\A}{\mathsf{A}}
\newcommand{\T}{\mathsf{T}}
\renewcommand{\H}{\mathsf{H}}
\renewcommand{\S}{\mathsf{S}}
\newcommand{\SU}{\mathsf{SU}}
\renewcommand{\O}{\mathsf{O}}
\newcommand{\N}{\mathsf{N}}
\DeclareMathOperator{\Rm}{Rm}
\DeclareMathOperator{\Ric}{Ric}
\DeclareMathOperator{\Ad}{Ad}

\DeclareMathOperator{\End}{End}
\DeclareMathOperator{\Id}{Id}
\DeclareMathOperator{\rank}{rank}
\DeclareMathOperator{\vol}{vol}
\DeclareMathOperator{\Gr}{Gr}
\DeclareMathOperator{\inj}{inj}
\newcommand{\abs}[1]{\ensuremath{\left|#1\right|}}
\let\subset\subseteq

\theoremstyle{plain}
\newtheorem{theorem}{Theorem}[section]
\newtheorem{proposition}[theorem]{Proposition}
\newtheorem{lemma}[theorem]{Lemma}
\newtheorem{corollary}[theorem]{Corollary}
\newtheorem{inttheorem}{\sc Theorem}

\newtheorem{intcorollary}[inttheorem]{\sc Corollary}
\theoremstyle{definition}
\newtheorem{definition}[theorem]{Definition}
\theoremstyle{remark}
\newtheorem{remark}[theorem]{Remark}
\numberwithin{equation}{section}

\begin{document}
\begin{abstract}
  We prove a precompactness result for homogeneous Einstein metrics on the set of homogeneous torus bundles over a common base. This set can contain infinite families of compact simply connected homogeneous torus bundles over a common simply connected base with toral fibres of fixed dimension.
\end{abstract}
\maketitle

\section{Introduction}
A Riemannian manifold $(M,g)$ is called Einstein if it satisfies $\Ric(g) = \lambda g$ for some $\lambda\in\R$. Major existence results for Einstein metrics have been obtained via constructions involving holonomy, bundle structures, or symmetry (see \cite{Bes08, Wan99, Joy00, And10, Spa11, Wan12, Jab23} for surveys). In particular, the last approach yields a rich collection of examples on homogeneous spaces (see, e.g., \cite{Jen73, Wan82, WZ86}).\smallskip

The recent proof of the Alekseevskii conjecture~\cite{BL23} can be considered a classification of (non-compact) homogeneous Einstein manifolds with negative Einstein constant, see also \cite{Lau10,LL14}. Moreover, Ricci flat homogeneous Einstein manifolds are flat by~\cite{AK75}. Thus to complete the classification of homogeneous Einstein manifolds, it remains to consider the case of positive Einstein constant, which is expected to be the most difficult case. Recall that by the Bonnet--Myers theorem such homogeneous spaces are necessarily compact.\smallskip

On compact homogeneous spaces, the Einstein equation reduces to a system of algebraic equations. Although this makes the problem more tractable, the classification of compact simply connected homogeneous Einstein manifolds is known only up to dimension 12 (see \cite{BK06}). Beyond this, there are several existence results for certain infinite families (see, e.g., \cite{ACS10,CN19,LW25}). Moreover, a systematic approach to the study of homogeneous Einstein metrics has been achieved by using their variational characterization as critical points of the Hilbert functional (see \cite{WZ86,BWZ04,Boh04}). This then leads to stability questions (see, e.g., \cite{Lau22, LW22, SS25}). Another line of investigation concerns the moduli space of invariant Einstein metrics on compact homogeneous spaces, leading to the Finiteness Conjecture of \cite{BWZ04} (see \cite{BF25} for recent progress). We refer the reader to \cite{BK23} for a list of open problems in the area, in particular \cite[Problem~1]{BK23}, which motivated the investigation of the present article.\smallskip

One of the main challenges in the classification of homogeneous Einstein manifolds comes from the presence of infinite families of torus bundles in a fixed dimension arising from different embeddings of the isotropy group into a fixed group. This phenomenon is illustrated by the Aloff--Wallach spaces $W_{k,l} := \SU(3)/ \S^1_{k,l}$, where $\S^1_{k,l}$ is a circle embedded in the maximal torus $\T^2$ of $\SU(3)$ with rational slope $(k,l)$. Each $W_{k,l}$ is a simply connected circle bundle over $\SU(3)/\T^2$. Moreover, by \cite{Wan82, Nik04}, they all admit $\SU(3)$-invariant Einstein metrics. A key question is whether the existence of an Einstein metric in such a family depends on a member or only on the family itself. If the latter holds, the existence question could be treated in a unified way, reducing the classification to a finite question in each dimension (see \cite[Problem 1]{BK23}). This is indeed possible if the simplicial complex of the base is non-contractible (see~\cite{Boh04}) which is crucial for the classification up to dimension~12.\smallskip

Our general setup is as follows. Let $\H\subsetneq\G$ be compact connected semisimple Lie groups with $\G/\H$ simply connected and almost effective. Let $\N_\G(\H)$ be the normalizer of $\H$ in $\G$ and let $\A$ be a maximal torus in $\N_\G(\H)/\H$, with $\dim\A\ge2$. For fixed \mbox{$0< \ell < \dim\A$}, we consider homogeneous spaces $\G/\H\tilde\A$, where $\tilde\A$ is a compactly embedded subtorus with \mbox{$\dim\tilde\A = \dim\A-\ell$}. Note that $\G/\H\tilde\A$ is (the total space of) an $\ell$-torus bundle over the base $\G/\H\A$, with fibre $\A/\tilde\A$. We say that the homogeneous torus bundle $\G/\H\tilde\A$ is of \emph{submersion type} if any $\G$-invariant metric is additionally right $\A$-invariant, a condition which is generically true (see Theorem~\ref{intthm:generic}). By extending this family of smooth torus bundles to include \emph{locally homogeneous spaces}, we endow the family of (locally) homogeneous torus bundles $\G/\H\tilde\A$ with a natural Grassmannian topology (see Section~\ref{subsec:moduli_space}).\smallskip

\begin{inttheorem}\label{intthm:precompactness}
  Let $(M^n_i=\G/\H\tilde\A_i,g_i)_{i\in\mathbb{N}}$ be a sequence of $\G$-homogeneous Einstein torus bundles of submersion type, with fixed Einstein constant $1$. Then, up to passing to a subsequence, $(\G/\H\tilde\A_i,g_i)$ converges to a (locally) $\G$-homogeneous Einstein torus bundle $(M^n_\infty=\G/\H\tilde\A_\infty,g_\infty)$ in the Grassmannian topology.
\end{inttheorem}

We refer the reader to Remark \ref{rem:convergence_Grassmannian_topology} for an elaboration on the convergence of $(\G/\H\tilde\A_i,g_i)$ in the Grassmannian topology. The key point in proving Theorem~\ref{intthm:precompactness} is to exclude the possibility that the Einstein metrics $g_i$ degenerate, i.e., that they collapse or escape to infinity. In order to do this, we obtain uniform and explicit estimates for Einstein metrics on homogeneous torus bundles of submersion type, which to the best of our knowledge are novel. Recall that the compactness result by~\cite{BWZ04} for homogeneous Einstein metrics for a fixed homogeneous space does not provide explicit estimates. Therefore, it cannot be applied to sequences of homogeneous spaces.\smallskip

Note that the limit space $\G/\H\tilde\A_\infty$ is globally homogeneous if and only if $\tilde\A_\infty$ is compactly embedded in $\A$. For example, for a sequence of Aloff--Wallach spaces $W_{k_i,l_i}$, this is true if ${k_i}/{l_i}$ converges to a rational number.\smallskip

We remark that Theorem~\ref{intthm:precompactness} follows from a more general result in which we relax the assumption that the manifolds $\G/\H\tilde\A_i$ are of submersion type by requiring only that the metrics $g_i$ are additionally right $\A$-invariant (see Theorem~\ref{T:ThmA-more-general}). Nevertheless, it is expected that the class of homogeneous spaces of submersion type is itself large. For example, the Aloff--Wallach spaces $W_{k,l}$ are all of submersion type except for those that are isomorphic to $W_{1,0}$ or $W_{1,1}$. Our next result confirms this expectation for all such families of homogeneous torus bundles.

\begin{inttheorem}\label{intthm:generic}
  With respect to the Grassmannian topology, the subclass of (locally) homogeneous spaces of submersion type contains a dense open subset of full measure.
\end{inttheorem}

The following corollary of Theorems~\ref{intthm:precompactness} and~\ref{intthm:generic} highlights the significance of homogeneous spaces of submersion type within a given family.

\begin{intcorollary}
  For a family of homogeneous spaces of the form $\G/\H\tilde\A$, the existence of Einstein metrics on all spaces of submersion type implies the existence of Einstein metrics on all members of the family.
\end{intcorollary}
In the context of a unified approach to the existence question on a given family of homogeneous spaces $\G/\H\tilde\A$, we point out that the existence of Einstein metrics on manifolds which are not of submersion type does not in general imply the existence on those of submersion type, see e.g., \cite[Example 5.13 and the paragraph following its proof]{Boh05}. Indeed, in general the space of $\G$-invariant metrics on a non-submersion type manifold has greater dimension than the space of $\G$-invariant metrics on a manifold of submersion type in the same family. We also remark that in the latter case the dimension of the space of $\G$-invariant metrics is independent of the choice of a particular submersion type manifold in the family.\smallskip

While studying existence questions for Einstein metrics on the total space of principal bundles, Riemannian submersions and O'Neill's formulae naturally play an essential role. This can be seen, for instance, in the expression of the Ricci tensor of the total space in terms of the Ricci tensor of the fibre and base, along with the O'Neill tensor. This expression becomes more tractable when the fibres are tori (see Section~\ref{subsec:invariant_metrics}, cf.~\cite{WZ90}).\smallskip

Before proceeding, we briefly comment on the notion of convergence considered here. First, convergence in the Grassmannian topology places no restrictions on the topology of the underlying manifolds. This is illustrated by the sequence of Aloff--Wallach spaces $W_{i,i+1}$, which converges to $W_{1,1}$ in the Grassmannian topology, even though the manifolds $W_{i,i+1}$ and $W_{1,1}$ are not homeomorphic (see \cite[Lem.~3.3]{AW}).\smallskip

Next, convergence of the invariant Einstein metrics implies, in particular, that the curvature tensors and their covariant derivatives vary continuously (see the proof of Theorem~\ref{T:ThmA-more-general} for a discussion of this type of continuity for the Ricci tensor). However, the injectivity radius and volume need not vary continuously along the sequence. To see this, let $g_i$ be a sequence of invariant Einstein metrics on $W_{i,i+1}$, with Einstein constant $1$. By Theorem~\ref{intthm:precompactness}, $g_i$ converges to an Einstein metric on $W_{1,1}$. Since the curvature remains uniformly bounded while the underlying topology changes in the limit, it follows that $\inj_{g_i}(W_{i, i+1}) \to 0$ (see, e.g., \cite[Thm.~9.8]{AH}). Furthermore, $\vol_{g_i}(W_{i, i+1})\to 0$ by \cite[Prop.~2]{Wan82}. Finally, we remark that the convergence in Theorem~\ref{intthm:precompactness} is an instance of the \emph{algebraic convergence} described in \cite[Sec.~6.3]{Lau12}.\smallskip

We now outline the proof of Theorem~\ref{intthm:precompactness}. First, the assumption of submersion type enables us to descend to a sequence of metrics on the common base of the torus bundles. By using the Gap theorem (see \cite{BLS19}) and the Einstein condition in conjunction with a compactness result for a sequence of homogeneous metrics collapsing with bounded curvature (see \cite{KPS26}, cf.~\cite{Ped19}), we obtain uniform bounds along the sequence for the metric eigenvalues in the horizontal directions. By proving uniform bounds on Lie algebraic data along the sequence of total spaces (see Section~\ref{sec:homtorbun}), we are then able to obtain uniform bounds on the metric eigenvalues in the fibre directions. This establishes convergence in the Euclidean topology of the sequence of metrics to a metric on the limit manifold. Finally, combining the metric convergence with the continuity of the relevant algebraic quantities, we obtain continuity of the curvature tensors and conclude that the limit metric is Einstein.\smallskip

The proof of Theorem~\ref{intthm:generic} is representation-theoretic. We examine the isotropy representation of $\G/\H\A$, and thereby describe the set of manifolds that are not of submersion type in terms of kernels of weights of the $\A$-action. This allows us to deduce that, in the Grassmannian topology, this set is nowhere dense and of measure zero.\smallskip

This article is organized as follows. Section~\ref{sec:comphomsp} contains general background information about compact homogeneous spaces. In Section~\ref{sec:homtorbun} we discuss homogeneous torus bundles, their moduli space, and invariant metrics on them. In Section~\ref{sec:closedness_proof}, we derive estimates for homogeneous Einstein metrics and prove Theorem~\ref{intthm:precompactness}. Section~\ref{sec:generic_char} contains the proof of Theorem~\ref{intthm:generic}.

\section{General background about compact homogeneous spaces}\label{sec:comphomsp}
Let $M = \G/\H$ be a homogeneous space with $\G$ and $\H$ compact connected Lie groups. We write $\g$, $\h$ for the Lie algebras of $\G$, $\H$ respectively. Let $Q$ be a biinvariant metric on $\g$, and let $\m$ be the $Q$-orthogonal complement of $\h$ in $\g$. Note that $\m$ is identified with the tangent space to $\G/\H$ at the origin. The isotropy group $\H$ acts on $\m$ by the adjoint action. Under this action, $\m$ can be decomposed into a sum of $Q$-orthogonal, irreducible $\H$-modules.
\begin{equation}\label{eqn:isotropydecomp}
  \m = \m_1 \oplus \cdots \oplus \m_q.
\end{equation}
Two $\Ad(\H)$-modules $\m_i$ and $\m_j$ in $\m$ are said to be equivalent if there exists a linear map $L:\m_i \to \m_j$ such that $\Ad(h)\circ L = L\circ\Ad(h)$ for each $h\in \H$. In general, the decomposition~\eqref{eqn:isotropydecomp} is not unique, however, the number $q$ of submodules and the dimensions $d_i:= \dim\m_i$ are independent of the specific choice of decomposition.\smallskip

We denote the set of all irreducible decompositions of $\m$ by $\F^{\G/\H}$. Note that by \cite[Sec.~4.3]{Boh04}, $\F^{\G/\H}$ is a compact set. We use the notation $\varphi=(\m_1,\ldots, \m_q)$ to denote a decomposition. For brevity and when no confusion arises, we may, by abuse of language, refer to an $\Ad(\H)$-invariant decomposition of $\m$ simply as a decomposition for the homogeneous space $\G/\H$. We will use the symbol $\m_0$ for the trivial $\Ad(\H)$-submodule of $\m$, that is,
\[ \m_0 := \{X \in \m \mid [Z, X] = 0\,\,\, \forall\,Z\in \h\}. \]
By \cite[Lem.~4.27]{Boh04}, $\m_0$ is the Lie algebra of a compact connected complement of $\H$ inside the normalizer $\N_\G(\H)$.\smallskip

Let $\{e_\alpha\}$ be a $Q$-orthonormal basis for $\m$ that is adapted to the decomposition $\varphi=(\m_1,\ldots, \m_q)$. Let
\begin{equation*}
  [ijk]_\varphi := \sum_{e_\alpha\in\m_i}\sum_{e_\beta\in\m_j}\sum_{e_\gamma\in\m_k}Q([e_\alpha, e_\beta], e_\gamma)^2, \quad \text{ for any } 1\leq i,j,k \leq q.
\end{equation*}
Note that these non-negative \emph{bracket constants} $[ijk]_\varphi$ depend on the choice of decomposition $\varphi$ for $\m$, but are independent of the choice of $Q$-orthonormal bases for $\m_i$, $\m_j$ and $\m_k$, and are pairwise symmetric in the indices $i,j,k$. By \cite[Sec.~4.3]{Boh04}, the function $\varphi\mapsto [ijk]_\varphi$ is continuous on $\F^{\G/\H}$.\smallskip

Now, let $g$ be a $\G$-invariant metric on the homogeneous space $\G/\H$. Then $g$ can be identified with an $\Ad(\H)$-invariant inner product on $\m$. Further, we can select a decomposition $\varphi = (\m_1, \dots, \m_q) \in \F^{\G/\H}$ with respect to which $g$ is diagonal, that is,
\begin{equation}\label{eqn:invmetric}
  g = x_1\,Q|_{\m_1} + \cdots + x_q\,Q|_{\m_q}.
\end{equation}

We will call such a decomposition $\varphi$ a {\it diagonalizing decomposition} for $g$. Note that the diagonalizing decomposition for $g$ may not be unique. Given a metric $g$ of the form~\eqref{eqn:invmetric}, its Ricci tensor is given as follows.

\begin{proposition} \label{prop:RicInvMetricGZ}
  If $\{e_\alpha\}$ is a $Q$-orthonormal basis for $\m$ adapted to the decomposition $\m = \m_1\oplus \cdots \oplus \m_q$, and $X\in \m_i$, $Y \in \m_j$, with $i \neq j$, then the Ricci tensor $\Ric(g)$ of the metric~\eqref{eqn:invmetric} satisfies:
  \begin{align*}
    \Ric(g)(X, X) &= \sum_{e_\alpha \in \m_i}\|[X,e_\alpha]_\h\|_Q^2 + \sum_{r,s}\frac{x_i^2 - 2x_r^2 + 2x_rx_s}{4x_rx_s}\sum_{e_\alpha \in \m_r}\|[X,e_\alpha]_{\m_s}\|_Q^2,\\
    \Ric(g)(X,Y) &= \sum_{r,s}\frac{x_ix_j - 2x_r^2 + 2x_rx_s}{4x_rx_s}\sum_{e_\alpha \in \m_r} Q([X,e_\alpha]_{\m_s}, [Y,e_\alpha]_{\m_s}).
  \end{align*}

\end{proposition}
\begin{proof}
  We refer to \cite[Prop.~1.14 (a),(b)]{GZ02} (cf.~\cite[(7.38)]{Bes08}). Although the formulae there are for cohomogeneity one manifolds, the corresponding terms for a homogeneous space can be deduced by subtracting the second fundamental form contribution, i.e., terms that involve derivatives. Then the result follows directly after adapting notation.
\end{proof}

\section{Homogeneous torus bundles}\label{sec:homtorbun}
In this section, we introduce the class of homogeneous torus bundles that forms the setting for our study of Einstein metrics and their convergence. We begin by discussing the algebraic structure of these bundles. We then endow this class of homogeneous torus bundles with a topology and establish several continuity results and uniform estimates that play a central role in the proofs of our main results. Finally, we recall some basic facts about the Ricci tensor of Riemannian submersions, specialized to the case of homogeneous torus bundles.\smallskip

Let $\G$ be a compact connected semisimple Lie group, let $\H$ be a closed connected semisimple subgroup of $\G$, and assume that $\G/\H$ is simply connected and almost effective. We let $\N_\G(\H)$ denote the normalizer of $\H$ in $\G$ and $\A$ a maximal torus in $\N_\G(\H)/\H$ and set $a:=\dim\A\ge2$. Note that $\a$, the Lie algebra of $\A$, is a maximal abelian subalgebra of $\m_0$, where $\m_0$ is as in Section~\ref{sec:comphomsp}. Denote by $\n$ the $Q$-orthogonal complement of $\h\oplus\a$ in $\g$, and note that $\n$ is identified with the tangent space to $\G/\H\A$ at the origin.\smallskip

Now, let $0<\ell<a$, and let $\tilde\A\subseteq\A$ be a compact subtorus with $\dim\tilde\A=a-\ell$. Denote by $\tilde\a$ its Lie algebra, and let $\t$ be the $Q$-orthogonal complement of $\tilde\a$ in $\a$. Let $\T$ be the ($\ell$-dimensional) connected subgroup of $\A$ with Lie algebra $\t$, which can be identified with $\A/\tilde\A$. Since $\A \subset \N_\G(\H)$, this gives rise to a homogeneous $\ell$-torus bundle:
\begin{equation}\label{eqn:homtorusbundle}
  \T \to \G/\H\tilde\A \overset{\pi}{\to} \G/\H\A.
\end{equation}
Observe that since $\G/\H$ is simply connected and $\tilde\A$ is connected, the total space $\G/\H\tilde\A$ is simply connected.\smallskip

Now let $\tilde\n$ denote the $Q$-orthogonal complement of $\h\oplus\tilde\a$ in $\g$. Then $\tilde\n$ is identified with the tangent space to $\G/\H\tilde\A$ at the origin. Thus, at the Lie algebra level, the bundle~\eqref{eqn:homtorusbundle} corresponds to the $Q$-orthogonal splitting:
\begin{equation}\label{eqn:Torus-Bunlde-Algebraic}
  \g = (\h \oplus \tilde\a) \oplus \tilde\n = (\h \oplus \tilde\a) \oplus \t \oplus \n.
\end{equation}
Observe that $\t$ and $\n$ are identified with tangent spaces to fibre and base respectively. A $\G$-invariant metric $g$ on $\G/\H\tilde\A$ can be identified with an $\Ad(\H\tilde\A)$-invariant inner product on $\tilde\n = \t \oplus \n$, also denoted by $g$ by abuse of notation.

\begin{remark}\label{Rem:Locally_Hom}
  If $\tilde\A$ is not closed in $\A$, then $M=\G/\H\tilde\A$ is not Hausdorff. In particular, $M$ is not a smooth manifold, and hence cannot be realized as a globally homogeneous space. Nevertheless, there exists a well-defined \emph{local factor space}, obtained roughly by quotienting a neighbourhood of the identity in $\G$ by a neighbourhood of the identity in $\H\tilde\A$. This local factor space is a topological manifold admitting a real analytic structure, unique up to $\mathcal{C}^\omega$-diffeomorphism (see \cite[Prop.~6.1]{Ped20} for a precise definition), which we by abuse of notation continue to denote by $\G/\H\tilde\A$. Furthermore, it is a \emph{locally homogeneous space}, i.e., for every pair of points $p,q\in M$, there exists $\varepsilon>0$ such that the neighbourhoods $B_\varepsilon(p)$ and $B_\varepsilon(q)$ are isometric. In this setting, \eqref{eqn:homtorusbundle} becomes a locally homogeneous torus bundle.

  The Lie algebraic description of these spaces is identical to that of the globally homogeneous case. In particular, the decomposition~\eqref{eqn:Torus-Bunlde-Algebraic} remains valid. Furthermore, with a slight abuse of language, we shall refer to $\Ad(\H\tilde\A)$-invariant inner products on $\tilde\n$ as $\G$-invariant metrics on $\G/\H\tilde\A$.
\end{remark}

\subsection{Moduli space of homogeneous torus bundles}\label{subsec:moduli_space}
In this subsection, we describe a natural topology on the moduli space of (locally) homogeneous torus bundles with fixed base and fibre dimension. Further, we prove several convergence results that will be used in the proof of our main theorems.

We denote by $\K(\G,\H,\ell)$ the collection of all (locally) homogeneous spaces of the form $\G/\H\tilde\A$, appearing as the total space of the $\ell$-torus bundle~\eqref{eqn:homtorusbundle}. Unless explicitly stated otherwise, the term \emph{homogeneous space} always refers to a globally homogeneous space. We equip $\K(\G,\H,\ell)$ with a topology via the natural bijection
\[
  \Gr_{a-\ell}(\a) \to \K(\G,\H,\ell), \quad
  \tilde\a \mapsto \G/\big(\H\cdot\exp(\tilde\a)\big),
\]
where $\Gr_{a-\ell}(\a)$ denotes the Grassmannian of $(a-\ell)$-dimensional subspaces of $\a$. We call this topology on $\K(\G,\H,\ell)$ the \emph{Grassmannian topology}. Since $\Gr_\ell(\a)$ is canonically isomorphic to $\Gr_{a-\ell}(\a)$, a sequence $\G/\H\tilde\A_n$ converges to $\G/\H\tilde\A$ in $\K(\G,\H,\ell)$ if and only if $\tilde\a_n^\perp$ converges to $\tilde\a^\perp$ in $\Gr_\ell(\a)$. Finally, we denote by $V_\ell(\a)$ the Stiefel manifold of $Q$-orthonormal $\ell$-frames in $\a$ and write \mbox{$\rho:V_\ell(\a) \to \Gr_\ell(\a)$} for the canonical bundle map which sends a given $Q$-orthonormal $\ell$-frame to the $\ell$-dimensional subspace spanned by it.

\begin{proposition}\label{prop:closed_groups_dense}
  The collection of all (globally) homogeneous spaces in $\K(\G,\H,\ell)$ forms a dense subset.
\end{proposition}
\begin{proof}
  Recall that the Lie algebra of any $\ell$-torus admits a $Q$-orthogonal basis $\mathcal{B}=\{e_1,\dots,e_\ell\}$ such that \mbox{$\exp(t\cdot e_i)=1$} if and only if $t\in\Z$. Note that if a Lie subalgebra $\t\subset\a$ admits a $Q$-orthonormal $\ell$-frame $\mathcal{F}$ whose coordinates with respect to $\mathcal{B}$ are elements of $\Q$, then $\exp(\t)$ is closed. It is sufficient to show that any $Q$-orthonormal $\ell$-frame $\mathcal{F}'$ can be arbitrarily closely approximated by a $Q$-orthonormal $\ell$-frame having rational coordinates. Since $V_\ell(\a)\cong\O(a)/\O(a-\ell)$ and $\Gr_\ell(\a)\cong\O(a)/\O(\ell)\O(a-\ell)$, this approximation follows immediately from the fact that $\O(a, \Q)$ is dense in $\O(a)$.
\end{proof}

We will now introduce notation related to isotropy decompositions of spaces in $\K(\G,\H,\ell)$. Accordingly, let $\G/\H\tilde\A\in \K(\G,\H,\ell)$ be a (locally) homogeneous space with $\tilde\n=\t\oplus\n$ as in~\eqref{eqn:Torus-Bunlde-Algebraic}. Consider an $\Ad(\H\A)$-irreducible decomposition $\varphi\in\F^{\G/\H\A}$ of $\n$, given by
\begin{equation}\label{eqn:isodecGHss}
  \n=\n_1\oplus\cdots\oplus\n_r.
\end{equation}
Since the $\Ad(\H\A)$-action on $\t$ is trivial, every $\Ad(\H\A)$-irreducible decomposition of $\t$ is naturally identified with an $\ell$-frame $\hat\varphi\in V_\ell(\a)$ satisfying $\rho(\hat\varphi)=\t$, and conversely every such frame determines a decomposition of $\t$. Consequently, an $\Ad(\H\A)$-irreducible decomposition of $\tilde\n=\n\oplus\t$ is determined by a pair
\[
  \tilde\varphi=(\varphi,\hat\varphi)\in\F^{\G/\H\A}\times V_\ell(\a),
  \quad \rho(\hat\varphi)=\t,
\]
and gives rise to
\begin{equation}\label{eqn:isodecGHt}
  \tilde\n
  =\n_1\oplus\cdots\oplus\n_r
  \oplus\n_{r+1}\oplus\cdots\oplus\n_{r+\ell}.
\end{equation}
This decomposition will be used repeatedly throughout the paper. Note that the decomposition~\eqref{eqn:isodecGHt} is $\Ad(\H\tilde\A)$-invariant but, in general, not $\Ad(\H\tilde\A)$-irreducible.\smallskip

We now prove a convergence result for (locally) homogeneous spaces in $\K(\G,\H,\ell)$ together with the associated decompositions. Although straightforward, this result is central to the arguments leading to the proof of Theorem~\ref{intthm:precompactness}.

\begin{proposition}\label{prop:Basis_diagonalize}
  Consider a sequence $(M_n, \tilde\varphi_n)_{n\in\mathbb{N}}$ in $\K(\G,\H,\ell)\times(\F^{\G/\H\A}\times V_\ell(\a))$, where for each $n$, $\tilde\varphi_n$ is a decomposition for $M_n=\G/\H\tilde\A_n$. Then there exists $(M,\tilde\varphi)$ with $M = \G/\H\tilde\A\in \K(\G,\H,\ell)$ and a decomposition $\tilde\varphi\in \F^{\G/\H\A}\times V_\ell(\a)$ for $M$ such that, up to passing to a subsequence, $M_n \to M$ and $\tilde\varphi_n \to \tilde\varphi.$
\end{proposition}
\begin{proof}
  By compactness of $\K(\G,\H,\ell)$ and $\F^{\G/\H\A}\times V_\ell(\a)$, there exists $M \in \K(\G,\H,\ell)$ such that $M_n \to M$ and there exists $\tilde\varphi = (\varphi, \hat\varphi)\in \F^{\G/\H\A}\times V_\ell(\a)$ such that $(\varphi_n, \hat\varphi_n) = \tilde\varphi_n \to \tilde\varphi$. It remains to show that $\tilde\varphi$ is a decomposition for $M$. That is, we need to show that the span of the $\ell$-frame $\hat\varphi$ equals the element $\t \in \Gr_{\ell}(\a)$ corresponding to $M$. In order to see this, note that since $\tilde\varphi_n$ is a decomposition for $M_n$, for each $n$ we have $\rho(\hat\varphi_n) = \t_n$, where $\t_n \in \Gr_{\ell}(\a)$ is the element corresponding to $M_n$. Since $\rho$ is a bundle map, and since $\t_n \to \t$, we can take limits, and have $\rho(\hat\varphi) = \t$.
\end{proof}
Recall that by \cite[Sec.~4.3]{Boh04}, on a fixed homogeneous space, the bracket constants depend continuously on the isotropy decomposition. In the next proposition, we observe that for the homogeneous spaces under consideration, this continuity remains valid even when the space itself varies.

\begin{proposition}\label{prop:BracketsCont}
  For each $1\leq i, j, k\leq r+\ell$, the function $\tilde\varphi \mapsto [ijk]_{\tilde\varphi}$ is a continuous function on $\F^{\G/\H\A}\times V_\ell(\a)$.
\end{proposition}
\begin{proof}
  Suppose $\tilde\varphi_n$ is a sequence in $\F^{\G/\H\A}\times V_\ell(\a)$, converging to $\tilde\varphi \in \F^{\G/\H\A}\times V_\ell(\a)$. We need to show that $[ijk]_{\tilde\varphi_n}$ converges to $[ijk]_{\tilde\varphi}$. For the sake of notation, we write $\tilde\varphi_n = (\n_1^{(n)}, \dots, \n_r^{(n)}, \n_{r+1}^{(n)}, \dots, \n_{r+\ell}^{(n)})$ and $\tilde\varphi = (\n_1, \dots, \n_r, \n_{r+1}, \dots, \n_{r+\ell})$. By possibly relabelling, we can assume that $\n_i^{(n)} \to \n_i$ for each $i$. Hence, for each $i$, there exists a $Q$-orthonormal basis for $\n_i^{(n)}$ which converges to a $Q$-orthonormal basis for $\n_i$. For each $n$ and for each $i,j,k$, the function defining the bracket constant $[ijk]_{\tilde\varphi_n}$ is a polynomial function in the elements of the chosen $Q$-orthonormal bases of $\n_i^{(n)}$, $\n_j^{(n)}$, $\n_k^{(n)}$, so it is straightforward to see it converges to the function defining the bracket constant $[ijk]_{\tilde\varphi}$. Moreover, the bracket constants are independent of choice of $Q$-orthonormal basis used to define them. This completes the proof.
\end{proof}

\subsection{Further algebraic properties of homogeneous torus bundles} In this subsection, we investigate some algebraic data of homogeneous torus bundles. We begin with bracket constants.

\renewcommand{\theenumi}{\alph{enumi}}\renewcommand{\labelenumi}{(\theenumi)}\leftmargini.78\leftmargini
\begin{proposition}\label{prop:bracconst}
  Let $\G/\H\tilde\A\in\K(\G,\H,\ell)$ and $\tilde\varphi\in \F^{\G/\H\A} \times V_\ell(\a)$ be a decomposition of $\tilde \n$ as in~\eqref{eqn:isodecGHt}. Then the bracket constants satisfy:
  \begin{enumerate}
    \item\label{item:bracconst_i} $[ijk]_{\tilde\varphi} = 0$ if $i,j\in\{r+1, \dots, r+\ell\}$,
    \item\label{item:bracconst_ii} $[ijk]_{\tilde\varphi} = 0$ if $i\in \{r+1, \dots, r+\ell\}$, $j,k\in \{1, \dots, r\}$, $j\neq k$.
  \end{enumerate}
\end{proposition}
\begin{proof}
  The claim in~\eqref{item:bracconst_i} is true because if $i,j\in\{r+1, \dots, r+\ell\}$, then $\n_i,\n_j \subseteq \t$, so $[X, Y]=0$ whenever $X\in \n_i$, $Y\in \n_j$. For~\eqref{item:bracconst_ii}, note that since $\tilde\varphi\in \F^{\G/\H\A} \times V_\ell(\a)$, it follows that $\n_j$ (resp. $\n_k$) is an $\Ad(\H\A)$-module, so in particular, $[X, Y]\in \n_j$ whenever $X\in \t$ and $Y\in \n_j$.
\end{proof}

Before proceeding, let us fix some notation. Let $\z(\m_0)$ be the centre of the Lie algebra $\m_0$. Then $\a=\z(\m_0)\oplus\a_0$, where $\a_0$ is a Cartan subalgebra of the semisimple part of $\m_0$. Moreover, let $\m_0^{\rm rs}$ be the direct sum of the root spaces of $\m_0$ associated to the Cartan subalgebra $\a_0$. Let $\m_0^\a$ be the maximal submodule of $\n$ on which $\a$ acts trivially, but $\h$ does not. Finally, let $\m(\h,\a)$ be the maximal submodule of $\n$ on which neither $\h$ nor $\a$ acts trivially. Then, the following is an $\Ad(\H\A)$-invariant (and $Q$-orthogonal) decomposition of $\g$:
\begin{equation}\label{eqn:rough_decom}
  \g=\h\oplus \a \oplus \m_0^{\rm rs} \oplus \m_0^\a\oplus\m(\h,\a).
\end{equation}
Note that the summands above are inequivalent as $\Ad(\H\A)$-modules, so any $\tilde\varphi \in \F^{\G/\H\A} \times V_\ell(\a)$ must respect the decomposition~\eqref{eqn:rough_decom}. Also note that since $\g$ is semisimple, if $\h=0$ then $\z(\m_0) = 0$. Furthermore, in this case~\eqref{eqn:rough_decom} simplifies to $\g = \a \oplus \m_0^{\rm rs}$.\smallskip

For later purposes, we analyze two cases depending on whether $\m(\h,\a)$ is a trivial submodule or not. We begin by recalling the following general result.
\begin{lemma}\label{lem:BK_Product}\cite[Lem.~5.21]{BK23}
  Let $\G/\H$ be an almost effective compact homogeneous space and suppose that $\m = \m_1 \oplus \m_2,$ where $\m_1,\m_2$ are $\Ad(\H)$-invariant, $Q(\m_1,\m_2)=0$, $[\m_1,\m_1]\subset \h\oplus \m_1$, and $[\m_2,\m_2]\subset \h\oplus \m_2.$ Then $\G/\H = \G_1/\H_1 \times \G_2/\H_2,$ where $\h_i=[\m_i,\m_i]_\h$, and $\g_i=\h_i\oplus\m_i$, for $i=1, 2$.
\end{lemma}

Using the above lemma, we deduce a topological splitting for the case where $\m(\h,\a)=0$.
\begin{lemma}\label{lem:product_space}
  Assume $\G/\H\A$ is a homogeneous space such that $\G/\H$ is almost effective and $\m(\h,\a)=0$. Then $\G=\G_1\times \G_2$ and $\G/\H\A=\G_1/\A\times\G_2/\H$, where $\rank\A=\rank\G_1$.
\end{lemma}
\begin{proof}
  Since $\m(\h,\a)=0$, from equation~\eqref{eqn:rough_decom}, we obtain the following $\Ad(\H\A)$-invariant decomposition of $\g$:
  \begin{equation}\label{eqn:rough_decomp1}
    \g=\h\oplus \a \oplus \m_0^{\rm rs}\oplus \m_0^\a.
  \end{equation}
  By Lemma~\ref{lem:BK_Product}, it suffices to prove that
  \begin{equation}\label{eqn:root_prod}
    [\m_0^{\rm rs}, \m_0^{\rm rs}]_{\h\oplus \a}=\a\quad \text{and}\quad [\m_0^{\rm rs}, \m_0^{\rm rs}]\subseteq \h \oplus \a \oplus \m_0^{\rm rs},
  \end{equation}
  and that
  \begin{equation}\label{eqn:non_root_prod}
    [\m_0^\a, \m_0^\a]_{\h\oplus \a}=\h\quad \text{and}\quad [\m_0^\a, \m_0^\a]\subseteq \h\oplus\a \oplus\m_0^\a.
  \end{equation}
  Since $\m_0^{\rm rs}$ is the direct sum of the root spaces of $\a_0$, the second relation of~\eqref{eqn:root_prod} follows immediately, as does $[\m_0^{\rm rs}, \m_0^{\rm rs}]_{\h\oplus \a}\subseteq \a$. To prove $[\m_0^{\rm rs}, \m_0^{\rm rs}]_{\h\oplus \a} = \a$, let $\a_1=[\m_0^{\rm rs}, \m_0^{\rm rs}]_{\h\oplus \a}$ and let $\a_1^\perp$ denote its $Q$-orthogonal complement in $\a$. Then by~\eqref{eqn:rough_decomp1} we have
  \begin{align}
    [\g,\a^\perp_1] &= [\h,\a^\perp_1]+[\a,\a^\perp_1]+[ \m_0^{\rm rs},\a^\perp_1]+[ \m_0^\a,\a^\perp_1] = [ \m_0^{\rm rs},\a^\perp_1] = 0.
  \end{align}
  Since $\G/\H$ is assumed to be almost effective, this implies that $\a_1^\perp=0$. This proves $[\m_0^{\rm rs}, \m_0^{\rm rs}]_{\h\oplus \a}=\a$, which completes the proof of~\eqref{eqn:root_prod}.\smallskip

  Now, observe that the Jacobi identity implies $[\a, [\m_0^\a,\m_0^\a]]=0$. Since $\a_0 \subset \a$ acts non-trivially on the sum of root spaces $\m_0^{\rm rs}$, this implies that $[\m_0^\a,\m_0^\a]_{\m_0^{\rm rs}} = 0$. We conclude that $[\m_0^\a, \m_0^\a]\subseteq \h\oplus\a\oplus \m_0^\a$, which is the second relation of~\eqref{eqn:non_root_prod}.

  Next, note that since $[\a,\m_0^\a]=0$, we have $ [\m_0^\a, \m_0^\a]_{\h\oplus \a}\subseteq\h$. Define $\h_1=[\m_0^\a, \m_0^\a]_{\h\oplus \a}$. Then we have $[\m_0^\a, \m_0^\a]\subseteq \h_1\oplus \m_0^\a$. Let $\h_1^\perp$ denote the $Q$-orthogonal complement of $\h_1$ in $\h$; we claim that $\h_1^\perp = \{0\}$. Indeed, if $X \in \h_1^\perp$ and $Y_1, Y_2 \in \m_0^{\a}$, then by definition we have $Q([Y_1,Y_2], X) = 0$. By skew-symmetry of $\operatorname{ad}$ with respect to the bi-invariant metric $Q$, this yields $Q([X,Y_1], Y_2) = 0$. Thus $X$ acts trivially on $\m_0^{\a}$. Since $X$ acts trivially on $\a \oplus \m_0^{\rm rs} = \m_0$, by~\eqref{eqn:rough_decomp1} we then see that $X$ acts trivially on $\m$. This contradiction to the almost effective action of $\G$ shows that $\h_1^\perp$ must be trivial. This completes the proof of~\eqref{eqn:non_root_prod}.\smallskip

  As noted above, by Lemma~\ref{lem:BK_Product}, the claimed splitting of $\G/\H\A$ follows.
\end{proof}

Next, we consider the case where $\m(\h,\a)\neq 0$ and obtain uniform estimates for the non-zero bracket constants, cf.\ Proposition~\ref{prop:bracconst}.
\begin{lemma}\label{lem:E1-nonzero-coeff}
  Let $\G, \H, \A$ be such that either $\h = 0$ or $\m(\h,\a)\neq 0$. Then, there exist constants $\Lambda\geq \delta>0$ such that the following holds. For each decomposition $\tilde\varphi \in \F^{\G/\H\A} \times V_\ell(\a)$, and for each $r+1 \leq i \leq r+\ell$, we have
  \begin{enumerate}
    \item\label{item:E1-nonzero-coeff_i} $[ikk]_{\tilde\varphi} \leq \Lambda$ for all $1\leq k\leq r$,
    \item\label{item:E1-nonzero-coeff_ii} $[ijj]_{\tilde\varphi} \geq \delta$ for some $1\leq j \leq r$.
  \end{enumerate}
\end{lemma}
\begin{proof}
  The proof of~\eqref{item:E1-nonzero-coeff_i} follows immediately by Proposition~\ref{prop:BracketsCont}, since the space $\F^{\G/\H\A} \times V_\ell(\a)$ is compact.\smallskip

  In order to prove~\eqref{item:E1-nonzero-coeff_ii}, fix a decomposition $(\n_1, \dots, \n_{r+\ell}) = \tilde\varphi \in \F^{\G/\H\A} \times V_\ell(\a)$. We claim that for each $r+1 \leq i \leq r+\ell$, there exists $1\leq j \leq r$ such that $[ijj]_{\tilde\varphi}>0$. For each such $i$, there are two cases to consider, either $Q(\n_i,\a_0)\neq 0$ or $\n_i \subset \z(\m_0)$. If $Q(\n_i,\a_0)\neq 0$, then $[ijj]_{\tilde\varphi}>0$ for some $\n_j\subset \m_0$ which is a root space corresponding to the Cartan subalgebra $\a_0 \subset \m_0$. If $\h = 0$, then since $\g$ is semisimple, $\n_i$ is necessarily contained in $\a_0$ for each $i$, hence the claim.

  If $\m(\h,\a) \neq 0$, we also need to consider the case of indices $i$ where $Q(\n_i,\a_0)= 0$. Accordingly, assume that $\n_i\subset \z(\m_0)$ and let $0\neq X_i \in \n_i$. We claim that there is $X\in \m(\h,\a)$ such that $[X_i,X]\neq 0$. Indeed, if $X_i \in \z(\m_0)$ acts trivially on $\m(\h,\a)$, then it must also act trivially on $\g$, which contradicts the fact that $\g$ is semisimple. Thus in this case as well, there exists some $\n_j\subset \m(\h,\a)$ with $[ijj]_{\tilde\varphi}>0$. This proves the claim in this case.

  The result then follows by Proposition~\ref{prop:BracketsCont}, and compactness of the space $\F^{\G/\H\A} \times V_\ell(\a)$.
\end{proof}

\renewcommand{\theenumi}{\roman{enumi}}\renewcommand{\labelenumi}{(\theenumi)}
\subsection{Invariant metrics on homogeneous torus bundles}\label{subsec:invariant_metrics}
In this subsection, we investigate $\G$-invariant metrics and their Ricci curvature on (locally) homogeneous spaces $\G/\H\tilde\A$. We recall some known facts about these metrics, specifically those for which the bundle~\eqref{eqn:homtorusbundle} is a Riemannian submersion.

\begin{definition}
  A $\G$-invariant metric $g$ on $\G/\H\tilde\A$ is called an \emph{$\A$-submersion metric} with respect to the bundle~\eqref{eqn:homtorusbundle} if it satisfies that
  \begin{enumerate}
    \item $\t$ and $\n$, i.e., the tangent spaces to fibre and base are $g$-orthogonal, and
    \item $\check g := g|_\n$ is $\Ad(\H\A)$-invariant.
  \end{enumerate}
\end{definition}

The terminology comes from the fact that an $\A$-submersion metric $g$ is invariant under the right action by $\A$ on $\G/\H\tilde\A$ and $\check g$ defines a $\G$-invariant metric on the base $\G/\H\A$ of the torus bundle, such that $\pi: (\G/\H\tilde\A, g) \to (\G/\H\A, \check g)$ is a Riemannian submersion. In this case, we can write
\[ g = \hat g + \check g, \]
where $\hat g$ is an $\Ad(\H\tilde\A)$-invariant inner product on $\t$, which is necessarily flat.\smallskip

The $\G$-invariant, right $\A$-invariant metrics on $\G/\H\tilde\A$ correspond precisely to the $\Ad(\H\A)$-invariant scalar products on $\tilde\n$. Therefore, we have the following characterization of these metrics. Although its proof is straightforward, we state it as a proposition for future reference.
\begin{proposition}\label{prop:G-invariant-right-invariant-Submersion}
  A $\G$-invariant metric $g$ on $\G/\H\tilde\A \in \K(\G,\H,\ell)$ is an $\A$-submersion metric if and only if there exists a decomposition $\tilde\varphi\in \F^{\G/\H\A} \times V_\ell(\a)$ as in~\eqref{eqn:isodecGHt} with respect to which $g$ is diagonal, that is,
  \begin{equation}\label{eqn:invmetricGHt}
    g = x_1\,Q|_{\n_1} + \cdots + x_{r+\ell}\,Q|_{\n_{r+\ell}}.
  \end{equation}
\end{proposition}
Now we provide some details about the convergence of the sequence $(\G/\H\tilde\A_i,g_i)$ in Theorem~\ref{intthm:precompactness}.
\begin{remark}
\label{rem:convergence_Grassmannian_topology}
  Consider a sequence $(\G/\H\tilde\A_i,g_i)_{i\in\mathbb N}$, where each $\G/\H\tilde\A_i$ is in $\K(\G,\H,\ell)$ and $g_i$ is an $\A$-submersion metric. By Proposition~\ref{prop:G-invariant-right-invariant-Submersion} there is sequence $\tilde\varphi_i$ in $\F^{\G/\H}\times V_\ell(\a)$ and a sequence $x^{(i)}\in\R_+^{r+\ell}$ corresponding to the metrics $g_i$. By Proposition \ref{prop:Basis_diagonalize}, $\G/\H\tilde\A_i$ converges to some $\G/\H\tilde\A$ in the topology of $\K(\G,\H,\ell)$ and $\tilde\varphi_i$ converges in $\F^{\G/\H\A} \times V_\ell(\a)$ to $\tilde\varphi$, a decomposition for $\G/\H\tilde\A$. Assume $x^{(i)}$ converges to $x$ in $\R_+^{r+\ell}$. Let $g$ be the metric on $\G/\H\tilde\A$ whose coordinates (with respect to the limit decomposition $\tilde\varphi$) are given by the entries of $x$. By abuse of language, in this situation we will say that $(\G/\H\tilde\A_i,g_i)$ converges to $(\G/\H\tilde\A,g)$ in the Grassmannian topology.
\end{remark}
We now turn our attention to the Ricci tensor of $\A$-submersion metrics on $\G/\H\tilde\A$. By \cite[Lem.~3.5]{Ped19}, the fibres of~\eqref{eqn:homtorusbundle} are totally geodesic. Furthermore, under the $\Ad(\H\A)$-action, $\t$ is a trivial module while $\n$ is non-trivial. Hence, by $\Ad(\H\A)$-invariance of $g$, we have $\Ric_g(\n,\t) = 0$. As a result, the expressions in \cite[Prop.~9.36]{Bes08} simplify, and we have
\begin{align}
  \Ric_g(U,V) &= (AU, AV) \label{eqn:ricRiemSubmTotGeodFib-VV},\\
  \Ric_g(X,U) &= 0 \label{eqn:ricRiemSubmTotGeodFib-VH},\\
  \Ric_g(X,Y) &= \Ric_{\check g}(X,Y) - 2(A_X, A_Y) \label{eqn:ricRiemSubmTotGeodFib-HH},
\end{align}
for each $X, Y \in \n$, $U, V \in \t$, and
\begin{equation*}
  (AU, AV) := \sum_ig(A_{X_i}U, A_{X_i}V),\qquad (A_X, A_Y):= \sum_ig(A_XX_i, A_YX_i) = \sum_jg(A_XU_j, A_YU_j).
\end{equation*}
Here $\{U_j\}$ (resp. $\{X_i\}$) is a $g$-orthonormal basis for $\t$ (resp. $\n$), and $A$ denotes the O'Neill tensor, which measures how far the horizontal distribution $\n$ is from being integrable (see~\cite[Sec.~2]{ONei66} and \cite[Ch.~9.C]{Bes08} for more details). Also note that by \cite[Ch.~9]{Bes08}, we have
\begin{equation}\label{eqn:norm-A}
  |A|^2 = \sum_i (A_{X_i}, A_{X_i}) = \sum_j (A U_j, A U_j).
\end{equation}
%

\section{Uniform Estimates and Precompactness}\label{sec:closedness_proof}

The main goal of this section is to prove precompactness of homogeneous Einstein torus bundles. More precisely, we prove the following result, from which Theorem~\ref{intthm:precompactness} follows immediately.

\begin{theorem}\label{T:ThmA-more-general}
  Let $\G/\H\tilde\A_i\in \K(\G,\H,\ell)$ be a sequence of homogeneous spaces such that for each $i$, $\G/\H\tilde\A_i$ admits an $\A$-submersion Einstein metric $g_i$ with Einstein constant $1$. Then, up to subsequence, $\G/\H\tilde\A_i$ converges to a (locally) homogeneous space $\G/\H\tilde\A\in \K(\G,\H,\ell)$ and $g_i$ converges to an $\A$-submersion Einstein metric $g$ on $\G/\H\tilde\A$.
\end{theorem}
We now work towards proving Theorem~\ref{T:ThmA-more-general}. Recall from Proposition~\ref{prop:G-invariant-right-invariant-Submersion} that, for each $i$, the metrics $g_i$ appearing in Theorem~\ref{T:ThmA-more-general} can be diagonalized with respect to decompositions $\tilde\varphi_i \in \F^{\G/\H\A} \times V_{\ell}(\a)$. To prove the theorem, we will establish a series of uniform estimates on the eigenvalues of $\A$-submersion Einstein metrics diagonalized with respect to such decompositions. We begin by recalling the following theorem which can be extracted from the proof of \cite[Thm.~3.1]{KPS26} (cf.~\cite[Thm.~A]{Ped19}) and reformulated for our purposes.
\begin{theorem}\cite[Thm.~3.1]{KPS26} \label{Thm:Toral_Direction}
  \renewcommand{\K}{\mathsf{K}}
  Let $\G/\K$ be a compact connected homogeneous space with $\rank \K=\rank \N_\G(\K)$. Then, for every $C>0$, there exists a constant $\varepsilon>0$ such that the following holds. For any $\G$-invariant metric $g$ on $\G/\K$ with $|\sec_g|\leq C$, the metric coefficients of $g$ with respect to any diagonalizing decomposition are bounded below by~$\varepsilon$.
\end{theorem}
Now we generalize this theorem, by assuming bounded Ricci curvature instead of bounded sectional curvature.
\begin{proposition}\label{prop:Compactness_Fixed_Space}
  \renewcommand{\K}{\mathsf{K}}
  Let $\G/\K$ be a compact connected homogeneous space with $\rank\K=\rank\N_\G(\K)$. Then for every $L>1$, there exists a constant $C_0>0$ such that the following holds. For any $\G$-invariant metric $g$ on $\G/\K$ with $\|\Ric_g\|_g\leq L$ and $\Ric_g \geq 1$, the metric coefficients of $g$ with respect to any diagonalizing decomposition are bounded from below by $C_0^{-1}$ and from above by $C_0$.
\end{proposition}
\begin{proof}
  \renewcommand{\K}{\mathsf{K}}
  Let $\varphi\in \F^{\G/\K}$ be a diagonalizing decomposition for $g$, with respect to which $g=(x_1,\dots,x_q)$. The Gap Theorem \cite[Thm.~4 and the subsequent paragraph]{BLS19} applied to $\G/\K$ implies the existence of a constant $C = C (\dim (\G / \K))$ such that $\|\Rm_g\|_g \leq C\, \|\Ric_g\|_g$. For any pair of $g$-orthonormal vectors $\{X,Y\}$ in the tangent space to $\G/\K$ at the origin, we conclude
  \begin{equation}\label{eqn:gap2}
    |\sec_g(X, Y)| \leq \|\Rm_g\|_g \leq C\, \|\Ric_g\|_g \leq CL.
  \end{equation}
  Then, since $\rank \K=\rank \N_\G(\K)$, the lower bound on the metric coefficients follows from Theorem~\ref{Thm:Toral_Direction}.\smallskip

  Since $\Ric_g\geq 1$, it follows from the Bishop--Gromov theorem that there exists a constant $\tilde C$ such that $\vol_g(\G/\K)\leq \tilde C$. Since $\vol_g(\G/\K) = \vol_Q(\G/\K)\cdot \Pi_{i=1}^q x_i^{d_i}$, and the coefficients $x_i$ are uniformly bounded from below, the upper bound follows.
\end{proof}

As a consequence, we obtain uniform bounds on the metric eigenvalues in horizontal directions for an $\A$-submersion Einstein manifold in $\K(\G,\H,\ell)$.
\begin{corollary}\label{cor:base_bounded}
  There exists a constant $C_1>0$ such that for any homogeneous space $M\in \K(\G,\H,\ell)$ and any $\A$-submersion Einstein metric $g$ on $M$ with Einstein constant $1$, of the form~\eqref{eqn:invmetricGHt}, we have
  \[ C_1^{-1} \leq x_i \leq C_1, \quad \text{ for each } 1\leq i \leq r. \]
\end{corollary}
\begin{proof}
  Let $\{X_i\}$ be a $g$-orthonormal basis for $\n$, and $\{U_j\}$ a $g$-orthonormal basis for $\t$. Then, since by assumption, $g$ is Einstein with Einstein constant $1$, we have
  \[ \Ric_g(X_i,X_i)=1, \quad\text{and}\quad (A U_j, A U_j)=\Ric(U_j, U_j)=1, \]
  where the latter follows from equation~\eqref{eqn:ricRiemSubmTotGeodFib-VV}. It follows from equation~\eqref{eqn:ricRiemSubmTotGeodFib-HH} that $\Ric_{\check g}(X_i, X_i)\geq 1$. Indeed,
  \[ \Ric_{\check g}(X_i, X_i)=\Ric_g(X_i,X_i) + 2(A_{X_i}, A_{X_i})\geq 1, \]
  since $(A_{X_i}, A_{X_i})\geq 0$. Now, we obtain
  \begin{align*}
    \|\Ric_{\check g}\|_{\check g}
    \leq \sum_{i=1}^r \Ric_{\check g}(X_i, X_i)
    = \sum_{i=1}^r (\Ric_g(X_i,X_i) + 2(A_{X_i}, A_{X_i}))
    = \sum_{i=1}^r \Ric_g(X_i,X_i) + 2\sum_{j=r+1}^{r+\ell}(A U_j, A U_j)
    = r + 2\ell,
  \end{align*}
  where the second last equality, uses equation~\eqref{eqn:norm-A}. Since $\A$ is a maximal torus in $\N_\G(\H)/\H$, we have $\rank \H\A = \rank \N_\G(\H\A)$, and so the result follows by Proposition~\ref{prop:Compactness_Fixed_Space} applied to the homogeneous space $(\G/\H\A, \check g)$.
\end{proof}

The following result demonstrates that the metric eigenvalues of the fibre are uniformly bounded both from above and away from zero.
\begin{proposition}\label{prop:fiber_bounded}
  Let $\G,\H,\A$ be such that either $\h=0$ or $\m(\h,\a)\neq 0$. There exists a constant $C_2>0$ such that for any homogeneous space $M\in \K(\G,\H,\ell)$ and any $\A$-submersion Einstein metric $g$ on $M$ with Einstein constant $1$, of the form~\eqref{eqn:invmetricGHt}, we have
  \[ C_2^{-1} \leq x_i \leq C_2, \quad \text{ for each } r+1\leq i \leq r+\ell. \]
\end{proposition}
\begin{proof}
  Let $\tilde\varphi \in \F^{\G/\H\A}\times V_\ell(\a)$ be a diagonalizing decomposition for the metric $g$ and let $\{U_j\}$ be a $g$-orthonormal basis for $\t$, adapted to the decomposition $\tilde\varphi$. Now, fix $i$ with $r+1\leq i \leq r+\ell$. Using Proposition~\ref{prop:RicInvMetricGZ} to compute the Ricci tensor in the vertical directions, along with the assumption that the Einstein constant is one, we obtain
  \begin{equation}\label{eqn:RicVert}
    \Ric(U_i, U_i) = \frac{x_i}{4}\sum_{j=1}^{r} \frac{[ijj]_{\tilde\varphi}}{x_j^2} = 1.
  \end{equation}
  Recall from Lemma~\ref{lem:E1-nonzero-coeff}\eqref{item:E1-nonzero-coeff_ii} that, there exists $\delta>0$ such that $[ijj]_{\tilde\varphi}\geq \delta$ for some $1\leq j \leq r$. Moreover, by Corollary~\ref{cor:base_bounded}, we have $x_k \leq C_1$ for all $1\leq k \leq r$. Combining these two estimates with equation~\eqref{eqn:RicVert}, we obtain
  \begin{equation*}
    1 = \frac{x_i}{4} \sum_{k=1}^{r}\frac{[ikk]_{\tilde\varphi}}{x_k^2}
    \ge \frac{x_i}{4} \cdot \frac{[ijj]_{\tilde\varphi}}{x_j^2}
    \geq \frac{x_i}{4} \cdot \delta \cdot \frac{1}{C_1^2},
  \end{equation*}
  which yields the upper bound. The lower bound follows in a similar fashion from Lemma~\ref{lem:E1-nonzero-coeff}\eqref{item:E1-nonzero-coeff_i}, Corollary~\ref{cor:base_bounded}, and equation~\eqref{eqn:RicVert}.
\end{proof}
Having established these estimates, we are now ready to prove Theorem~\ref{T:ThmA-more-general}.
\begin{proof}[Proof of Theorem~\ref{T:ThmA-more-general}]
  To prove the theorem, we divide the argument into two cases.\smallskip

  \noindent{\it Case 1: For the homogeneous space $\G/\H\A$, either $\h=0$ or $\m(\h,\a)\neq 0$.}\\
  Let $(M_i, g_i, \tilde\varphi_i)$ be a sequence such that, for each $i$, $M_i = \G/\H\tilde\A_i \in \K(\G,\H,\ell)$ is a homogeneous space, $g_i$ is an $\A$-submersion Einstein metric on $M_i$, and $\tilde\varphi_i \in \F^{\G/\H\A} \times V_{\ell}(\a)$ is a diagonalizing decomposition for $g_i$.\smallskip

  It follows from Proposition~\ref{prop:Basis_diagonalize} that $(M_i, \tilde\varphi_i)$ converges, up to a subsequence, to $(M, \tilde\varphi)$, where $M \in \K(\G,\H,\ell)$ is a (locally) homogeneous space and $\tilde\varphi \in \F^{\G/\H\A} \times V_{\ell}(\a)$ is a decomposition for $M$. We now show that $g_i$ converges to an $\A$-submersion Einstein metric $g$ on $M$.\smallskip

  For convenience of notation, we write the eigenvalues of $g_i$ as $x^{(i)} = (x_1^{(i)},\dots,x_r^{(i)},x_{r+1}^{(i)},\dots,x_{r+\ell}^{(i)}) \in (\R_+)^{r+\ell}$. Corollary~\ref{cor:base_bounded} and Proposition~\ref{prop:fiber_bounded} imply that there exists a closed disc $D \subset (\R_+)^{r+\ell}$ such that $x^{(i)} \in D$ for all $i$. By compactness of $D$, we obtain $x^{(i)} \to x \in D$. Let $g$ be the metric on $M$ whose coordinates (with respect to the limit decomposition $\tilde\varphi$) are given by the entries of $x$. It follows by Proposition~\ref{prop:G-invariant-right-invariant-Submersion} that $g$ is an $\A$-submersion metric.\smallskip

  It remains to show that the Einstein condition is preserved in the limit. To this end, it suffices to show that the Ricci tensors of $(M_i, g_i, \tilde\varphi_i)$ converge to the Ricci tensor of $(M,g,\tilde\varphi)$. First note that the convergence $\tilde\varphi_i \to \tilde\varphi$ implies that $Q$-orthonormal bases adapted to $\tilde\varphi_i$ can be chosen so that they converge to a $Q$-orthonormal basis adapted to $\tilde\varphi$. This yields continuity of the underlying Lie algebraic data associated with the decompositions, as reflected in the continuity of the bracket constants in Proposition~\ref{prop:BracketsCont}. Consequently, the functions defining the Ricci tensor in Proposition~\ref{prop:RicInvMetricGZ} depend continuously on both the metric parameters and the decompositions. The convergence of the Ricci tensors follows by continuity. Hence $(M,g)$ is Einstein with Einstein constant 1.\smallskip

  \noindent{\it Case 2: For the homogeneous space $\G/\H\A$, we have $\m(\h,\a)=0$.}\\
  By Lemma~\ref{lem:product_space}, we have the splitting $\G = \G_1\times \G_2$, and $\G/\H\A = \G_1/\A \times \G_2/\H$, which induces a splitting $M_i=\G/\H\tilde\A_i = \G_1/\tilde\A_i\times \G_2/\H,$ satisfying
  \[ T_o (\G_1/\tilde\A_i) \cong \t_i \oplus \m_0^{\rm rs}, \qquad T_o (\G_2/\H) \cong \m_0^{\a}. \]
  Further, since $g_i$ is an $\Ad(\H\A)$-invariant metric, and the two submodules $\t_i \oplus \m_0^{\rm rs}$ and $\m_0^{\a}$ are $\Ad(\H\A)$-inequivalent, it follows that $g_i$ is a product metric. Thus $(M_i, g_i) \cong (\G_1/\tilde\A_i, g_i^{(1)}) \times (\G_2/\H, g_i^{(2)})$, where $g_i^{(1)}$ and $g_i^{(2)}$ are Einstein metrics on the respective factors. Note that $\G_1/\tilde\A_i$ converges to $\G_1/\tilde\A$, where $\tilde\A \subset \A$. We show that the sequence $g_i^{(1)}$ converges to an $\A$-submersion Einstein metric $g^{(1)}$ on $\G_1/\tilde \A$ and $g_i^{(2)}$ converges to an Einstein metric $g^{(2)}$ on $\G_2/\H$.\smallskip

  To prove the convergence of $g_i^{(1)}$, note that the semisimple part of $\A\subset\G_1$ is trivial, and $\a$ acts non-trivially on the sum of root spaces $\m_0^{\rm rs}$. Thus we are in Case~1, and obtain the desired convergence of $g_i^{(1)}$ to $g^{(1)}$.\smallskip

  To prove the convergence of $g_i^{(2)}$, first observe that $\rank \G_2=\rank \H$. Then, since the metrics are Einstein with Einstein constant $1$, using Proposition~\ref{prop:Compactness_Fixed_Space}, we deduce that the metric parameters are bounded uniformly from above and away from zero. Therefore, up to subsequence, the metrics $g_i^{(2)}$ converge to a metric $g^{(2)}$. Then, using continuity of the bracket constants and Ricci tensor as in Case 1, we see that the limit metric $g^{(2)}$ is also Einstein.\smallskip

  Noting that the product metric $g^{(1)}\oplus g^{(2)}$ is an $\Ad(\H\A)$-invariant Einstein metric completes the proof.
\end{proof}

\section{Subtori are generically of submersion type}\label{sec:generic_char}
In this section we recall the notion of \emph{submersion type} and review some properties of the isotropy representation of $\G/\H\A$. Thereafter we prove Theorem~\ref{intthm:generic} by giving an algebraic characterization of homogeneous spaces of submersion type.

\begin{definition}\label{def:generic}
  We say that a homogeneous space $\G/\H\tilde\A\in\K(\G,\H,\ell)$ is \emph{of submersion type} if every $\G$-invariant metric on it is an $\A$-submersion metric with respect to the torus bundle~\eqref{eqn:homtorusbundle}. A subtorus $\tilde\A\subset\A$ or respectively its Lie algebra $\tilde\a\subset\a$ is \emph{of submersion type} if the homogeneous space $M = \G/\H\tilde\A$ is.
\end{definition}

Recall the isotropy representation~\eqref{eqn:isodecGHss} of $\G/\H\A$, where each $\n_j$ is an irreducible $\Ad(\H\A)$-submodule of $\n$. By construction of $\A$, there cannot be any $\n_j$ on which $\H\A$ acts trivially. Indeed, any non-zero element in a $\Ad(\H\A)$-trivial summand would generate a 1-parameter subgroup of $\G$ that commutes with $\H\A$, contradicting the maximality of $\A$ in $\N_\G(\H)/\H$. Note that for every $a=\exp(x)\in\A$ and $h\in\H$ it holds that $hah^{-1}=\exp(\Ad_h(x)) = \exp(x) =a$ since $x\in\a\subset\m_0$. Therefore, the map $\H\times\A\to\H\A$, $(h,a)\mapsto ha$, is a surjective group homomorphism and every $\H\A$-module is in particular also an $\H\times\A$-module.\smallskip

To understand the isotropy representation, we examine its irreducible submodules, thought of as $\Ad(\H\times\A)$-modules. The following lemma can be deduced from general theory about representations of compact Lie groups (see, e.g., \cite[Ch.~II]{BD85}), however, for the convenience of the reader, we provide an elementary proof.
\begin{lemma}\label{lem:weights_of_irreps}
  Let $V$ be a real non-trivial irreducible $\H\times\A$-module, then either
  \begin{enumerate}
    \item \label{item:weights_of_irreps1}$\A$ acts trivially and $V$ is an irreducible $\H$-module, or
    \item \label{item:weights_of_irreps2} $V$ is of complex type and $\A$ acts by multiplication with complex units with respect to the corresponding complex structure. That is, there exists an $I\in\End(V)$ with $I^2=-\Id_V$ which commutes with the $\H\times\A$-action and there is a non-zero linear map $\mu:\a\to\R$, such that $\exp(x).v = e^{\mu(x)I}(v)$ for all $x\in\a$ and $v\in V$.
  \end{enumerate}
\end{lemma}
\begin{proof}
  We may assume that $\A$ acts non-trivially on $V$, since otherwise~\eqref{item:weights_of_irreps1} is immediate. Let $V_\A\subset V$ be an irreducible $\A$-submodule of $V$. Since $\A$ is a torus, the (real) representation is either trivial or two dimensional (see~\cite[Ch.~II~Prop.\,8.5]{BD85}). In the latter case, there exists a complex structure, i.e., an endomorphism $I\in\End(V_\A)$ such that $I^2=-\Id_{V_\A}$, and a weight $\mu:\a\to\R$, such that
  \[ \exp(x).v = e^{\mu(x)I}(v) = \cos(\mu(x))v + \sin(\mu(x))I(v), \quad\text{ for all } v\in V_\A \text{ and } x\in\a. \]
  Note, that for any $x\in\a$ with $\mu(x)\neq0$ we obtain $a:=\exp\left(\tfrac{\pi}{2\mu(x)}x\right)\in\A$ with $a.v=I(v)$ for all $v\in V_\A$.\smallskip

  We first decompose $V$ into irreducible $\A$-modules. Choose $v\in V_{\A}\setminus\{0\}$ and $h_1\in \H$ such that $h_1.v\in V\setminus V_{\A}$. Because the $\A$- and $\H$-actions commute, the subspace $h_1.V_{\A}$ is again an irreducible $\A$-module. Inductively we obtain elements $e=h_0,h_1,\dots,h_q\in \H$, with $q\ge 0$, such that \mbox{$V=h_0.V_{\A}\oplus h_1.V_{\A}\oplus \dots \oplus h_q.V_{\A}$}.\smallskip

  We now show that $V_\A$ is a non-trivial $\A$-module. Suppose, for contradiction, that $\A$ acts trivially on $V_\A$, in particular $\dim V_\A=1$. Since $V$ is a non-trivial irreducible $\H\times\A$-module and since every non-trivial irreducible representation of a connected Lie group has dimension at least $2$, it follows that $V_\A\neq V$, i.e., $q\ge1$. Furthermore, $\A$ acts trivially on each summand $h_i.V_\A$ and hence $\A$ acts trivially on $V$, contradicting our assumption.\smallskip

  Thus $\A$ acts non-trivially on each summand $h_i.V_{\A}$. As described above, we choose an $a\in\A$ such that $a.v=I(v)$ for $v\in V_0$ and extend the complex structure $I$ to $V$ by $I(v):=a.v$. Evidently, since the $\A$- and $\H$-actions are linear and commute, this extended $I$ satisfies $I^2=-\Id_V$ and commutes with the $\H\times\A$-action. The desired properties of $\mu$ and the description of the $\A$-action on $V$ follow directly.
\end{proof}
The following is an immediate consequence.
\begin{corollary}\label{cor:Torus_Action_Rest}
  Let $(V, \rho)$ be an irreducible representation of a torus $\A$ and assume that for a subtorus $\tilde\A$, the restricted representation $(V,\rho|_{\tilde\A})$ is non-trivial. Then for any $a\in \A$ there exists $\tilde a\in \tilde\A$ such that $\rho(a)=\rho(\tilde a)$.
\end{corollary}
The following lemma is our main technical result which gives an algebraic characterization of subtori of submersion type. In order to state it, we introduce the following notation.
Based on the description of the $\Ad(\H\A)$-action on $\n$ in Lemma~\ref{lem:weights_of_irreps} we define the following sets.
\begin{align*}
  \J_1 & :=\left\{\; \mu:\a\to\R \;\middle|\; \mu \text{ is a non-zero weight of an irreducible $\Ad(\H\A)$-submodule } \n_i\subseteq\n. \;\right\}\\
  \J_2 & :=\left\{ \{\mu,\nu\}\subseteq \J_1 \;\middle|\;
    \begin{gathered}
      \mu, \nu \text{ are weights of irreducible $\Ad(\H\A)$-submodules } \n_i, \n_j\subset \n,\\[-1ex]
      \text{that are equivalent as $\Ad(\H)$-modules, and } \mu\pm\nu\neq0.
    \end{gathered} \right\}
\end{align*}
Note that these sets are finite. Further note that if $(\mu,I)$ is a pair as in Lemma~\ref{lem:weights_of_irreps}\eqref{item:weights_of_irreps2}, then $(-\mu, -I)$ is another such pair. In other words, if $\mu \in \J_1$, then $-\mu \in \J_1$ as well.\smallskip

Now, fix $0<\ell<a$ and let $\E_\ell$ be the set of all $(a-\ell)$-dimensional subalgebras of $\a$ given by
\[
\E_\ell = \left\{\; \tilde\a\in \Gr_{a-\ell}(\a) \;\;\middle|\;\;
\tilde\a\subseteq \Big(\cup_{\mu\in\J_1}\ker(\mu)\Big) \bigcup
\Big(\cup_{\{\mu,\nu\}\in\J_2}\ker(\mu\pm\nu)\Big) \;\right\}.
\]
\begin{lemma}\label{lem:submersion_type}
 Let $\mathcal{G}_{\ell}$ be the set of $(a-\ell)$-dimensional subalgebras of $\a$ which are of submersion type. Then
  \[ \Gr_{a-\ell}(\a) \setminus \E_{\ell}\subseteq\mathcal{G}_{\ell}. \]
\end{lemma}
\begin{proof}
  \newcommand{\scprod}{$\langle\cdot,\cdot\rangle$}
  Let $\tilde\a\in \Gr_{a-\ell}(\a) \setminus \E_{\ell}$ and let $\tilde\A\subseteq \A$ be the corresponding subtorus. Using the notation from Section~\ref{sec:homtorbun}, we show that for any $\Ad(\H\tilde\A)$-invariant scalar product \scprod\ on $\tilde\n$, the $Q$-orthogonal splitting $\tilde\n=\t\oplus\n$ is \scprod-orthogonal and that the restriction \scprod$|_\n$ is $\Ad(\H\A)$-invariant.\smallskip

  To prove the \scprod-orthogonality of $\t$ and $\n$ we first show that $\n$ cannot contain any trivial $\Ad(\H\tilde\A)$-submodule, by examining the individual summands of the $Q$-orthogonal and $\Ad(\H\A)$-invariant splitting $\n=\m_0^{\rm rs}\oplus\m_0^\a\oplus\m(\h,\a)$ (see~\eqref{eqn:rough_decom}). By Lemma~\ref{lem:weights_of_irreps}, the modules $\m_0^\a$ and $\m(\h,\a)$ cannot contain any trivial $\Ad(\H)$-submodule and therefore also no trivial $\Ad(\H\tilde\A)$-submodule. Any $\H\A$-submodule $\n_i$ of $\m_0^{\rm rs}$ is a root space of $\A$ and has a weight $\mu:\a\to\R$. The assumption $\tilde\a\notin\E_\ell$ includes the condition $\tilde\a\not\subseteq\ker(\mu)$ which implies $\tilde\A$- and hence $\H\tilde\A$-irreducibility of $\n_i$. Finally, Schur's lemma and $\Ad(\H\tilde\A)$-invariance of \scprod\ yield $\langle\n_i,\t\rangle=0$, since the $\Ad(\H\tilde\A)$-action on $\t$ is trivial.\smallskip

  It remains to show that \scprod$|_\n$ is $\Ad(\H\A)$-invariant. To this end, let $\varphi=(\n_1,\dots,\n_r)$ be an $\Ad(\H\A)$-irreducible decomposition of $\n$. We prove that, for each $1\leq i\leq r$, the restriction $\langle\cdot,\cdot\rangle|_{\n_i}$ is $\Ad(\H\A)$-invariant and for any $1\leq i<j\leq r$ the restriction \scprod$|_{\n_i\oplus\n_j}$ is also $\Ad(\H\A)$-invariant.\smallskip

  For the $\Ad(\H\A)$-invariance of  $\langle\cdot,\cdot\rangle|_{\n_i}$ we first prove that it is a multiple of $Q$. In fact, we only need to show that $\n_i$ is $\Ad(\H\tilde\A)$-irreducible. The result then follows from the $\Ad(\H\tilde\A)$-invariance of $\langle\cdot,\cdot\rangle$ and Schur's lemma. If $\A$ acts trivially on $\n_i$, then $\n_i$ is clearly $\Ad(\H\tilde\A)$-irreducible. If $\A$ acts non-trivially on $\n_i$ with weight~$\mu$, then $\tilde\A$ also acts non-trivially on $\n_i$ since, by assumption, $\tilde\a\not\subseteq \ker(\mu)$. Now, let $V$ be an $\Ad(\H\tilde\A)$-irreducible space of $\n_i$. By Corollary~\ref{cor:Torus_Action_Rest}, for every $g\in \A$ there exists $\tilde g\in\tilde \a$ such that $g.v=\tilde g.v\in V$ for all $v\in V$. This shows that $V$ is an $\Ad(\H\A)$-submodule of $\n_i$. By $\Ad(\H\A)$-irreducibility of $\n_i$, it holds $V=\n_i$. Therefore, $\n_i$ is $\Ad(\H\tilde\A)$-irreducible.\smallskip

  To show the $\Ad(\H\A)$-invariance of \scprod$|_{\n_i\oplus\n_j}$, we have to consider three cases: either $\A$ acts trivially on $\n_i$ and $\n_j$, or on exactly one of them, or on neither of them. In the first case, there is nothing to show. In the second case, $\n_i$ and $\n_j$ are irreducible as $\Ad(\H\tilde\A)$-modules but not equivalent as such. Therefore, by $\Ad(\H\tilde\A)$-invariance of $\langle\cdot,\cdot\rangle$, Schur's lemma implies that $\langle \n_i,\n_j\rangle = 0$.\smallskip

  In the third case, by Lemma~\ref{lem:weights_of_irreps}, we have two weights $\mu,\nu\in\J_1$ and two complex structures $I$ and $J$ on the respective irreducible summands, such that $\exp(x)\cdot v = e^{\mu(x)I}(v)$ and $\exp(x)\cdot w = e^{\nu(x)J}(w)$, for all $x\in \a$, $v\in \n_i$, $w\in\n_j$. Now, either $\mu=\pm\nu$, or $\mu\neq\pm\nu$. First, assume that $\mu=\pm\nu$. Since, by assumption, $\tilde\a\not\subset\ker(\mu)=\ker(\nu)$, it follows that $\tilde\A$ acts non-trivially on both $\n_i$ and $\n_j$. Therefore, Corollary~\ref{cor:Torus_Action_Rest} implies that for any $x\in\a$, there exists $\tilde x\in \tilde\a$ such that $\pm\nu(x)=\mu(x) = \mu(\tilde x)=\pm\nu(\tilde x)$. This shows that $\exp(x)\cdot v = e^{\mu(x)I}(v) = e^{\mu(\tilde x)I}(v) = \exp(\tilde x)\cdot v$, for any $v\in\n_i$, and similarly $\exp(x)\cdot w=\exp(\tilde x)\cdot w$, for any $w\in\n_j$. Hence,
  \[ \langle \exp(x)\cdot v,\exp(x)\cdot w \rangle = \langle \exp(\tilde x)\cdot v,\exp(\tilde x)\cdot w \rangle = \langle v,w\rangle, \]
  where the last step is simply the application of $\tilde\A$-invariance of $\langle\cdot,\cdot\rangle$.
  Now assume $\mu\neq\pm\nu$. The assumption $\tilde\a\not\subset\ker(\mu\pm\nu)\cup\ker(\mu)\cup\ker(\nu)$ yields that there exists $x\in\tilde\a$ such that $0\neq\mu(x)\neq\pm\nu(x)\neq0$. Without loss of generality assume that $\mu(x) = 2\pi$ and $\abs{\nu(x)}<2\pi$, by possibly interchanging the roles of $\mu$ and $\nu$ and multiplying $x$ by ${\mu(x)}^{-1}\in\R$. Then $\exp(x).v=v$ for all $v\in\n_i$ but $\exp(x).w\neq w$ for all non-zero $w\in\n_j$. This shows that $\n_i$ and $\n_j$ are inequivalent irreducible $\Ad(\H\tilde\A)$-modules and therefore $\langle\n_i,\n_j\rangle=0$.
\end{proof}

We can now prove that the subtori of submersion type are \emph{generic}.
\begin{proof}[Proof of Theorem~\ref{intthm:generic}]
  For any $\mu\in\J_1$ and $\{\rho,\nu\}\in\J_2$ define the following hyperplanes in $\a$:
  \begin{equation*}
    V_\mu            := \ker(\mu), \qquad
    V_{\rho\nu}^\pm  := \ker(\rho \pm \nu).
  \end{equation*}
  Note that there are natural inclusions of the Grassmannians $\Gr_{a-\ell}(V_\mu),\Gr_{a-\ell}(V^{\pm}_{\rho\nu}) \subset \Gr_{a-\ell}(\a)$, corresponding to the subspace inclusions $V_\mu,V^\pm_{\rho\nu}\subset\a$. It follows from the definitions that the set $\E_\ell$ can be written as
  \[
    \E_\ell = \Big(\cup_{\mu\in\J_1} \Gr_{a-\ell} (V_\mu)\Big) \bigcup \Big(\cup_{\{\rho,\nu\}\in\J_2} \Gr_{a-\ell}(V^\pm_{\rho\nu})\Big).
  \]
  This shows that $\E_\ell$ is a finite union of lower-dimensional closed submanifolds of $\Gr_{a-\ell}(\a)$. Hence, by Lemma~\ref{lem:submersion_type}, the interior of $\mathcal{G}_\ell$ is an open and dense subset of $\Gr_{a-\ell}(\a)$, of full measure. This completes the proof.
\end{proof}
\subsection*{Acknowledgements}
We would like to thank Christoph B\"ohm for suggesting this line of enquiry and for countless helpful discussions and insights. We also thank Francesco Pediconi for comments and suggestions.\smallskip

{AMK is supported by the IIT Bombay Seed Grant (RD/0525-IRCCSH0-024).}
{VPNW was funded by the Deutsche Forschungsgemeinschaft (DFG, German Research Foundation) under Germany's Excellence Strategy EXC 2044/2 -- 390685587, Mathematics M\"unster: Dynamics--Geometry--Structure and by the CRC 1442 “Geometry: Deformations and Rigidity” of the DFG.}
{MZ acknowledges support from the Deutsche Forschungsgemeinschaft (DFG, German Research Foundation) under Germany's Excellence Strategy -- EXC 2121 ``Quantum Universe'' -- 390833306.}


\begin{thebibliography}{00}
  \bibitem{AK75}{
    {\sc D. V. Alekseevski\u i}, 
    {\sc B. N. Kimel$'$fel$'$d}, 
    {\em Structure of homogeneous Riemannian spaces with zero Ricci curvature}, 
    {Funkcional. Anal. i Prilo\v zen.} 
    {\bf 9} 
    {(1975)}, 
    {no. 2}, 
    {5--11} 
    {(Russian)}. 
  }

  \bibitem{AW}{
    {\sc S. I. Aloff}, 
    {\sc N. R. Wallach}, 
    {\em An infinite family of distinct {$7$}-manifolds admitting positively curved Riemannian structures}, 
    {Bull. Amer. Math. Soc.} 
    {\bf 81} 
    {(1975)}, 
    {93--97}. 
  }

  \bibitem{And10}{
    {\sc M. T. Anderson}, 
    {\em A survey of Einstein metrics on 4-manifolds}, 
    {in \em Handbook of geometric analysis III}, 
    {1--39}, 
    {Adv. Lect. Math. (ALM)} 
    {\bf 14}, 
    {Int. Press, Somerville, MA}, 
    {2010}. 
  }

  \bibitem{AH}{
    {\sc B. Andrews}, 
    {\sc C. Hopper}, 
    {\em The Ricci flow in Riemannian geometry}, 
    {Lecture Notes in Mathematics}, 
    {2011}, 
    {Springer, Heidelberg}, 
    {2011}. 
  }

  \bibitem{ACS10}{
    {\sc A. Arvanitoyeorgos}, 
    {\sc I. Chrysikos}, 
    {\sc Y. Sakane}, 
    {\em Complete description of invariant Einstein metrics on the generalized flag manifold {${\rm SO}(2n)/{\rm U}(p)\times {\rm U}(n-p)$}}, 
    {Ann. Global Anal. Geom.} 
    {\bf 38} 
    {(2010)}, 
    {no. 4}, 
    {413--438}. 
  }

  \bibitem{Bes08}{
    {\sc A. L. Besse}, 
    {\em Einstein manifolds}, 
    {Classics in Mathematics}, 
    {Springer-Verlag, Berlin}, 
    {2008}. 
    {Reprint of the 1987 edition}. 
  }

  \bibitem{BF25}{
    {\sc R. G. Bettiol}, 
    {\sc H. Friedman}, 
    {\em Counting Homogeneous Einstein Metrics}, 
    {preprint}. 
    {\tt arXiv:2509.09830 [math.DG]}. 
  }

  \bibitem{Boh04}{
    {\sc C. B\"ohm}, 
    {\em Homogeneous Einstein metrics and simplicial complexes}, 
    {J. Differential Geom.} 
    {\bf 67} 
    {(2004)}, 
    {no. 1}, 
    {79--165}. 
  }

  \bibitem{Boh05}{
    \bysame, 
    {\em Non-existence of homogeneous Einstein metrics}, 
    {Comment. Math. Helv.} 
    {\bf 80} 
    {(2005)}, 
    {no. 1}, 
    {123--146}. 
  }

  \bibitem{BK06}{
    {\sc C. B\"ohm}, 
    {\sc M. M. Kerr}, 
    {\em Low-dimensional homogeneous Einstein manifolds}, 
    {Trans. Amer. Math. Soc.} 
    {\bf 358} 
    {(2006)}, 
    {no. 4}, 
    {1455--1468}. 
  }

  \bibitem{BK23}{
    \bysame, 
    {\em Homogeneous Einstein metrics and butterflies}, 
    {Ann. Global Anal. Geom.} 
    {\bf 63} 
    {(2023)}, 
    {no. 4}, 
    {Paper No. 29, 92}. 
  }

  \bibitem{BL23}{
    {\sc C. B\"ohm}, 
    {\sc R. A. Lafuente}, 
    {\em Non-compact Einstein manifolds with symmetry}, 
    {J. Amer. Math. Soc.} 
    {\bf 36} 
    {(2023)}, 
    {no. 3}, 
    {591--651}. 
  }

  \bibitem{BLS19}{
    {\sc C. B\"ohm}, 
    {\sc R. Lafuente}, 
    {\sc M. Simon}, 
    {\em Optimal curvature estimates for homogeneous Ricci flows}, 
    {Int. Math. Res. Not. IMRN} 
    {\bf 14} 
    {(2019)}, 
    {4431--4468}. 
  }

  \bibitem{BWZ04}{
    {\sc C. B\"ohm}, 
    {\sc M. Wang}, 
    {\sc W. Ziller}, 
    {\em A variational approach for compact homogeneous Einstein manifolds}, 
    {Geom. Funct. Anal.} 
    {\bf 14} 
    {(2004)}, 
    {no. 4}, 
    {681--733}. 
  }

  \bibitem{BD85}{
    {\sc T. Br\"ocker}, 
    {\sc \scshape T. t. Dieck}, 
    {\em Representations of compact Lie groups}, 
    {Springer}, 
    {1985}. 
  }

  \bibitem{CN19}{
    {\sc Z. Chen}, 
    {\sc Yu. G. Nikonorov}, 
    {\em Invariant Einstein metrics on generalized Wallach spaces}, 
    {Sci. China Math.} 
    {\bf 62} 
    {(2019)}, 
    {no. 3}, 
    {569--584}. 
  }

  \bibitem{GZ02}{
    {\sc K. Grove}, 
    {\sc W. Ziller}, 
    {\em Cohomogeneity one manifolds with positive Ricci curvature}, 
    {Invent. Math.} 
    {\bf 149} 
    {(2002)}, 
    {no. 3}, 
    {619--646}. 
  }

  \bibitem{Jab23}{
    {\sc M. Jablonski}, 
    {\em Homogeneous Einstein manifolds}, 
    {Rev. Un. Mat. Argentina} 
    {\bf 64} 
    {(2023)}, 
    {no. 2}, 
    {461--485}. 
  }

  \bibitem{Joy00}{
    {\sc D. Joyce}, 
    {\em Compact Manifolds with Special Holonomy}, 
    {Oxford Mathematical Monographs}, 
    {Oxford University Press}, 
    {Oxford}, 
    {2000}. 
  }

  \bibitem{Jen73}{
    {\sc G. Jensen}, 
    {\em Einstein metrics on principal fibre bundles}, 
    {J. Differential Geom.} 
    {\bf 8} 
    {(1973)}, 
    {599--614}. 
  }

  \bibitem{KPS26}{
    {\sc A. M. Krishnan}, 
    {\sc F. Pediconi}, 
    {\sc S. Sbiti}, 
    {\em Toral symmetries of collapsed ancient solutions to the homogeneous Ricci flow}, 
    {J. Lond. Math. Soc. (2)} 
    {\bf 113} 
    {(2026)}, 
    {no. 3}, 
    {Paper No. e70513, 52}. 
  }

  \bibitem{LL14}{
    {\sc R. Lafuente}, 
    {\sc J. Lauret}, 
    {\em Structure of homogeneous Ricci solitons and the Alekseevskii conjecture}, 
    {J. Differential Geom.} 
    {\bf 98} 
    {(2014)}, 
    {no. 2}, 
    {315--347}. 
  }
  \bibitem{Lau10}{
    {\sc J. Lauret}, 
    {\em Einstein solvmanifolds are standard}, 
    {Ann. of Math. (2)} 
    {\bf 172} 
    {(2010)}, 
    {no. 3}, 
    {1859--1877}. 
  }

  \bibitem{Lau12}{
    \bysame, 
    {\em Convergence of homogeneous manifolds}, 
    {J. Lond. Math. Soc. (2)} 
    {\bf 86} 
    {(2012)}, 
    {no. 3}, 
    {701--727}. 
  }

  \bibitem{Lau22}{
    \bysame, 
    {\em On the stability of homogeneous Einstein manifolds}, 
    {Asian J. Math.} 
    {\bf 26} 
    {(2022)}, 
    {no. 4}, 
    {555--584}. 
  }

  \bibitem{LW22}{
    {\sc J. Lauret}, 
    {\sc C. Will}, 
    {\em On the stability of homogeneous Einstein manifolds {II}}, 
    {J. Lond. Math. Soc. (2)} 
    {\bf 106} 
    {(2022)}, 
    {no. 4}, 
    {3638--3669}. 
  }

  \bibitem{LW25}{
    \bysame, 
    {\em Einstein metrics on aligned homogeneous spaces with two factors}, 
    {J. Lond. Math. Soc. (2)} 
    {\bf 111} 
    {(2025)}, 
    {no. 3}, 
    {Paper No. e70120, 26}. 
  }

  \bibitem{Nik04}{
    {\sc Yu. G. Nikonorov}, 
    {\em Compact homogeneous Einstein 7-manifolds}, 
    {Geom. Dedicata} 
    {\bf 109} 
    {(2004)}, 
    {7--30}. 
  }

  \bibitem{ONei66}{
    {\sc B. O'Neill}, 
    {\em The fundamental equations of a submersion}, 
    {Michigan Math. J.} 
    {\bf 13} 
    {(1966)}, 
    {459--469}. 
  }

  \bibitem{Ped19}{
    {\sc F. Pediconi}, 
    {\em Diverging sequences of unit volume invariant metrics with bounded curvature}, 
    {Ann. Global Anal. Geom.} 
    {\bf 56} 
    {(2019)}, 
    {no. 3}, 
    {519--553}. 
  }

  \bibitem{Ped20}{
    \bysame, 
    {\em A local version of the Myers--Steenrod theorem}, 
    {Bull. Lond. Math. Soc.} 
    {\bf 52} 
    {(2020)}, 
    {no. 5}, 
    {871--884}. 
  }

  \bibitem{SS25}{
    {\sc P. Schwahn}, 
    {\sc U. Semmelmann}, 
    {\em Einstein metrics, their moduli spaces and stability}, 
    {preprint}. 
    {\tt arXiv:2507.18463 [math.DG]}. 
  }

  \bibitem{Spa11}{
    {\sc J. Sparks}, 
    {\em Sasaki--Einstein manifolds}, 
    {in \em Surveys in differential geometry. Volume XVI. Geometry of special holonomy and related topics}, 
    {265--324}, 
    {Surv. Differ. Geom.} 
    {\bf 16} 
    {Int. Press, Somerville, MA}, 
    {2011}. 
  }

  \bibitem{Wan82}{
    {\sc M. Wang}, 
    {\em Some examples of homogeneous Einstein manifolds in dimension seven}, 
    {Duke Math. J.} 
    {\bf 49} 
    {(1982)}, 
    {no. 1}, 
    {23--28}. 
  }

  \bibitem{Wan99}{
    \bysame, 
    {\em Einstein metrics from symmetry and bundle constructions}, 
    {in \em Surveys Diff. Geom. VI: essays on Einstein manifolds}, 
    {287--325}, 
    {C. LeBrun}, 
    {Wang, M.} (eds.), 
    {Int. Press, Boston, MA}, 
    {1999}. 
  }

  \bibitem{Wan12}{
    \bysame, 
    {\em Einstein metrics from symmetry and bundle constructions: a sequel}, 
    {253--309}. 
    {in \em Differential geometry}, 
    {Adv. Lect. Math. (ALM)} 
    {\bf 22}, 
    {Int. Press, Somerville, MA}, 
    {2012}. 
  }

  \bibitem{WZ86}{
    {\sc M. Wang}, 
    {\sc W. Ziller}, 
    {\em Existence and non-existence of homogeneous Einstein metrics}, 
    {Invent. Math.} 
    {\bf 84} 
    {(1986)}, 
    {no. 1}, 
    {177--194}. 
  }

  \bibitem{WZ90}{
    \bysame, 
    {\em Einstein metrics on principal torus bundles}, 
    {J. Differential Geom.} 
    {\bf 31} 
    {(1990)}, 
    {no. 1}, 
    {215--248}. 
  }
\end{thebibliography}
\end{document}